\documentclass[11pt]{amsart}
\usepackage[latin1]{inputenc}
\usepackage{amsmath, amsfonts, amssymb, amsthm, mathtools, mathrsfs,xypic,bm}
\usepackage{enumerate}
\usepackage[numbers]{natbib}
\usepackage{setspace}

\usepackage[pagebackref,colorlinks,linkcolor=red,citecolor=blue,urlcolor=blue,hypertexnames=true]{hyperref}
\usepackage[pagebackref,hypertexnames=true]{hyperref}

\newtheorem{theorem}{Theorem}[section]
\newtheorem{lemma}[theorem]{Lemma}
\newtheorem{proposition}[theorem]{Proposition}
\newtheorem{corollary}[theorem]{Corollary}

\theoremstyle{definition}
\newtheorem{definition}[theorem]{Definition}
\newtheorem{example}[theorem]{Example}
\newtheorem{remark}[theorem]{Remark}

\newcommand{\LH}{L(H)}

\newcommand{\N}{\mathbb{N}}

\newcommand{\C}{\mathbb{C}}

\newcommand{\K}{\mathcal{K}}

\newcommand{\fset}{\mathcal{F}}		

\newcommand{\Cuntz}[1]{\mathcal{O}_{#1}}

\newcommand{\id}[1]{{\rm id}_{#1}}
\newcommand{\ltwo}[1]{H_{#1}}

\newcommand{\bdd}[2]{\mathcal{L}_{#1}(#2)}

\newcommand{\fsim}[2]{\,{\sim}_{#1,#2}\,}
\newcommand{\fleq}[2]{\,{\prec}_{#1,#2}\,}

\begin{document}

\title[homotopy symmetric $C^*$-algebras]{Deformations of nilpotent groups and  homotopy symmetric $C^*$-algebras}
\author{Marius Dadarlat}\address{MD: Department of Mathematics, Purdue University, West Lafayette, IN 47907, USA}\email{mdd@purdue.edu}
\author{Ulrich Pennig}\address{UP: Mathematisches Institut, Westf\"alische Wilhelms-Universit\"at M\"unster, Ein\-stein\-stra\ss e 62, 48149 M\"unster, Germany}\email{u.pennig@uni-muenster.de}	
\begin{abstract}
The homotopy symmetric $C^*$-algebras are those separable $C^*$-algebras for which one can unsuspend in E-theory.
We find a new simple condition that characterizes  homotopy symmetric nuclear $C^*$-algebras and  use it to show  that the property of being homotopy symmetric passes to nuclear $C^*$-subalgebras and it has a number of other significant permanence properties.

As an application, we show that if  $I(G)$ is the kernel
 of the trivial representation $\iota:C^*(G)\to \mathbb{C}$ for a countable discrete  torsion free nilpotent group $G$, then $I(G)$ is homotopy symmetric and hence the Kasparov group $KK(I(G),B)$ can be realized as the homotopy classes of asymptotic morphisms $[[I(G),B \otimes \K]]$
 for any separable $C^*$-algebra~$B$.
\end{abstract}
\thanks{M.D. was partially supported by NSF grant \#DMS--1362824}
\thanks{U.P. was partially supported by the SFB 878 -- ``Groups, Geometry \& Actions''}

\maketitle

\section{Introduction}

The intuitive idea of deformations of $C^*$-algebras was formalized by Connes and Higson
who introduced the concept of asymptotic morphism \cite{Con-Hig:etheory}.
These morphisms are at the heart of $E$-theory,  the universal bifunctor from
 the category of separable $C^*$-algebras to the category
 abelian groups which is homotopy invariant,  $C^*$-stable and half-exact.
Asymptotic morphisms have become important tools in other areas,
such as deformation quantization  \cite{MR1184061},  index theory \cite{Con:noncomm}, \cite{Trout-index},
 the Baum-Connes conjecture \cite{HigKas:BC},  shape theory  \cite{Dad:shape}, and classification theory of nuclear
 C*-algebras \cite{Ror:encyclopedia}.

An asymptotic morphism $(\varphi_t)_{t \in [0,\infty)}$ is given by a family of maps $\varphi_t \colon A \to B$ parametrized by $t \in [0,\infty)$ such that $t \mapsto \varphi_t$ is pointwise continuous and the axioms for  $*$-homomorphisms are satisfied \emph{asymptotically} for $t \to \infty$.
 There is a natural notion of homotopy based on asymptotic morphisms of the form $A \to B \otimes C[0,1]$.  Homotopy classes of asymptotic morphisms from the suspension $SA$ of $A$ to the stabilization of $SB$ provide a model for $E(A,B)$, i.e.\ $E(A,B) = [[SA, SB \otimes \K]]$, \cite{Con-Hig:etheory}.
A similar construction using completely positive contractive asymptotic morphisms yields a realization of $KK$-Theory, namely $KK(A,B)\cong [[SA, SB \otimes \K]]^{\rm cp},$ as proven by Houghton-Larsen and Thomsen \cite{Thomsen-Larsen}. $E$-theory factors through $KK$-theory
and the fact that the map $KK(A,B)\to E(A,B)$ is an isomorphism for nuclear $C^*$-algebras $A$ can be easily seen from the Choi-Effros theorem, which implies $[[A,B]] \cong [[A,B]]^{\rm cp}$.

Note that the introduction of the suspensions and the stabilization of $B$ are necessary to obtain a natural abelian group structure on $E(A,B)$. Equally important, but perhaps less apparent, is the fact that $SA$ becomes quasidiagonal,  as shown by Voiculescu \cite{Voi:qd}, and this property implicitly assures a large supply of almost multiplicative maps $SA \to \K$.
However, a deformation $A \to B \otimes \K$, without the suspensions, contains in principle more geometric information. Asymptotic morphisms of this form are a crucial tool in the classification theory of nuclear $C^*$-algebras. We are confronted with the dilemma of understanding $[[A,B \otimes \K]]$, while only $E(A,B) = [[SA, SB \otimes \K]]$ is computable using the tools of algebraic topology. The best case scenario in this situation is of course that the monoid homomorphism $[[A,B \otimes \K]] \to E(A,B)$ induced by the suspension map  is an isomorphism.

Under favorable circumstances, this is in fact true: It is shown in \cite{DadLor:unsusp} and \cite{Dad:unsusp-annalen} that for a connected compact metrizable space $X$ with basepoint $x_0 \in X$, we have $[[C_0(X \setminus x_0), B \otimes \K]] \cong E(C_0(X \setminus x_0), B)$. In particular, $K_0(X \setminus x_0) \cong [[C_0(X \setminus x_0), \K]]$. On the right hand side we can replace the compact operators $\K$ by $\bigcup_{n=1}^{\infty} M_n(\C)$. Thus, the reduced $K$-homology of $X$ classifies deformations of $C_0(X \setminus x_0)$ into matrices. Similar deformations of commutative $C^*$-algebras into matrix algebras appeared in condensed matter physics, where Kitaev \cite{Kitaev} proposed a classification of topological insulators via (real) $K$-homology.
We refer the reader to the recent work of Loring \cite{Loring:top-ins1} for further developments and additional references.

A full answer to the  question of unsuspending in E-theory was found in \cite{DadLor:unsusp}: The natural map $[[A, B \otimes \K]] \to E(A,B)$ is an isomorphism for all separable $C^*$-algebras $B$ if and only if $A$ is \emph{homotopy symmetric}, which means that $[[\id{A}]] \in [[A, A \otimes \K]]$ has an additive inverse or equivalently that $[[A, A \otimes \K]]$ is a group.

As nice as this condition is, it can be quite hard to check in practice. One of the two main results in this paper is Theorem~\ref{thm:QH_implies homotopy symmetric}, which shows that being homotopy symmetric is equivalent for separable nuclear $C^*$-algebras to a property which is significantly easier to verify. The proof of this theorem relies crucially on results of Thomsen \cite{Thomsen:discrete}.
The new property, which we call property \textrm{(QH)}, see Definition~\ref{def:strong_AHZ}~(i), makes sense for all separable $C^*$-algebras and has the important feature that it passes to $C^*$-subalgebras.  This allows us to exhibit new vast classes of homotopy symmetric $C^*$-algebras, see Theorem ~\ref{thm:permanence} and Corollary~\ref{cor:fields}.

Our second main result is Theorem~\ref{thm:AH2} which states that the augmentation ideal $I(G)=\ker(\iota:C^*(G)\to \C)$ for a torsion free countable discrete nilpotent group $G$ satisfies property \textrm{(QH)} and hence it is homotopy symmetric. In particular $[[I(G),B \otimes \K]]\cong KK(I(G),B)$ for any separable $C^*$-algebra $B$.
This confirms a conjecture of the first author \cite{Dadarlat-almost} for the class of nilpotent groups.

\section{Discrete asymptotic morphisms and Property \textrm{(QH)}}

\begin{definition}\label{def:ah-basic}
Let $A$, $(B_n)_n$ be separable $C^*$-algebras. A completely positive contractive (cpc)
\emph{discrete asymptotic morphism} from $A$ to $(B_n)_n$ is a sequence of
completely positive linear contractions $\{\varphi_n: A \to B_n\}_n$ such
that
$$
	\lim_{n\to \infty} \| \varphi_n(ab)-\varphi_n(a)\varphi_n(b)\|=0
$$
for all $a,b \in A$.
If, furthemore, each $\varphi_n$ is unital we say $(\varphi_n)$
is a \emph{ucp discrete asymptotic morphism}.
Two cpc discrete asymptotic morphisms $\{\varphi_n,\psi_n: A \to B_n\}_n$
are \emph{equivalent}, written $\varphi_n\approx\psi_n$, if $\lim_{n\to \infty} \| \varphi_n(a)-\psi_n(a)\|=0$ for all $a\in A$. They are \emph{unitarily equivalent},
written $(\varphi_n)\sim (\psi_n)$, if there is a sequence of unitaries
$u_n \in B_n$ such that $(\varphi_n)\approx (u_n\psi_n u_n^*)$. They are
\emph{homotopy equivalent}, written $(\varphi_n)\sim_h (\psi_n)$, if there is a
cpc discrete asymptotic morphism $\{\Phi_n \colon A \to C[0,1]\otimes B_n\}_n$
such that  $\Phi_n^{(0)}=\varphi_n$ and $\Phi_n^{(1)}=\psi_n$ for all
$n\geq 1$. Equivalent discrete asymptotic morphisms are homotopic
via the homotopy $\Phi_n^{(t)}=(1-t)\varphi_n+t\psi_n$. Homotopy of ucp asymptotic morphisms is defined similarly, by
requiring that $\Phi_n(1)=1$. The homotopy classes of cpc discrete asymptotic
morphisms $\varphi_n \colon A \to B$ from $A$ to a fixed $B = B_n$ will be
denoted by $[[A, B]]^{\rm cp}_{\N}$.  (Non-discrete) \emph{cpc asymptotic morphisms}
are defined completely analogously by replacing the indexing set $\N$ by
$[0, \infty)$ and considering cpc maps
$\varphi \colon A \to C_b([0, \infty), B)$.
The homotopy classes of
cpc asymptotic morphisms are denoted by $[[A,B]]^{\rm cp}$.
Similarly we denote by $[[A,B]]$ (respectively $[[A,B]]_{\N}$ in the discrete case) the homotopy classes of general asymptotic
morphisms, \cite{Con-Hig:etheory}, \cite{Thomsen:discrete}.
\end{definition}

\begin{remark} \label{rem:compacts_matrices}
If $A$ and $B$ are separable $C^*$-algebras,
then every cpc discrete asymptotic morphism
 $\{\varphi_n \colon A \to B \otimes \K\}_n$ is equivalent to some
 $\{\psi_n \colon A \to M_{m(n)}(B)\}_n$ obtained by compressing
$\varphi_n$ by a suitable sequence of projections $1_{M(B)} \otimes p_n \in M(B) \otimes \K$,
where $p_n \in \K$ converges strongly  to $1$.
\end{remark}

\begin{remark} \label{rem:category}
There is a category $\mathit{Asym}$ with objects separable $C^*$-algebras and with
homotopy classes of (non-discrete) asymptotic morphisms as morphisms. In particular,
there is a composition of asymptotic morphisms, which is well-defined up to homotopy
\cite{Con-Hig:etheory}.  There is a similar category $\mathit{Asym}^{\rm cp}$ based on cpc asymptotic morphisms.
There are no analogues of these categories for homotopy
classes of discrete asymptotic morphisms. Nevertheless, it was shown in \cite[Thm.\ 7.2]{Thomsen:discrete} that there is a well-defined pairing of the form
$[[A, B]]_{\N} \times [[B,C]] \to [[A,C]]_{\N}$,
$(x,y)\mapsto y\circ x$ such that $z\circ ( y\circ x)=(z\circ y)\circ x$
for $z\in [[C,D]]$.
The arguments of \cite[Thm.\ 7.2]{Thomsen:discrete} show that there is also a pairing
$[[A, B]]^{\rm cp}_{\N} \times [[B,C]]^{\rm cp} \to [[A,C]]^{\rm cp}_{\N}$ with similar properties.
\end{remark}

Any discrete asymptotic morphism $\{\varphi_n: A \to B_n\}_n$ induces a $*$-ho\-mo\-mor\-phism
\(
\bm{\Phi}:A \to \prod_n B_n/\bigoplus_n B_n\ .
\)

\begin{definition}\label{def:injective}
A discrete asymptotic morphism $\{\varphi_n: A \to B_n\}_n$ is called
\emph{injective} if the induced $*$-homomorphism
$\bm{\Phi}$ is injective, or equivalently if
\newline
$\limsup_n \|\varphi_n(a)\|=\|a\|$ for all $a\in A$.
\end{definition}

The cone over a $C^*$-algebra $B$ is defined as $CB=C_0(0,1]\otimes B$.
\begin{proposition}\label{prop:QH-equivalence}
For a separable $C^*$-algebra $A$ the following properties are equivalent.
\begin{itemize}
  \item[(i)] There is a null-homotopic
\emph{injective} cpc discrete asymptotic morphism $\{\eta_n \colon A \to \K \}_n$.
  \item[(ii)] There is a null-homotopic
\emph{injective} cpc discrete asymptotic morphism $\{\gamma_n \colon A \to \LH \}_n$.
  \item[(iii)] There is an injective $*$-homomorphism $\bm{\eta}: A\to\prod_n C\K/\bigoplus_n C\K$ which is liftable to a cpc map $\eta:A\to\prod_n C\K$.
   \item[(iv)] There is an injective $*$-homomorphism $\bm{\gamma}:A\to\prod_n C\LH/\bigoplus_n C\LH$ which is liftable to a cpc map $\gamma: A\to\prod_n C\LH$.
\end{itemize}
\end{proposition}
\begin{proof}
  (i) $\Rightarrow$ (ii) and (iii) $\Rightarrow$ (iv) are obvious since $\K\subset \LH$.
  The implications  (i) $\Rightarrow$ (iii) and (ii) $\Rightarrow$ (iv) are  are also straightforward.

 (iv) $\Rightarrow$ (i) Let $\gamma$ be as in (iv) with components $\gamma_n$.
 Since $\bm{\gamma}$ is injective, we must have $\limsup \|\gamma_n(a)\|=\|a\|$ for all $a\in A$.
 Let $(a_i)_i$ be a dense sequence in $A$. For each $i$ there is a strictly increasing sequence $(m(i,n))_n$ such that
 $\|\gamma_{m(i,n)}(a_i)\|\geq \|a_i\|-1/n,$ for all $n\geq 1$.
 If we define $\gamma'_n=\gamma_{m(1,n)}\oplus \gamma_{m(2,n)}\oplus \cdots \oplus \gamma_{m(n,n)}$, then
 $\lim\|\gamma'_n(a)\|=\|a\|$ for all $a\in A$. Therefore we may assume that $(\gamma_n)$ had this property in the first place.

 Since $A$ is separable, there is a separable $C^*$-algebra $D\subset \LH$ such that $\gamma_n(A)\subset D$ for all $n\geq 1$.
 Consider an injective $*$-homomorphism $j:C_0(0,1] \to C_0(0,1]\otimes C_0(0,1]$,  for example $j(f)(s,t)=f(\min(s,t))$.
 Since $C_0(0,1]\otimes D$ is quasidiagonal by \cite{Voi:qd}, there is a cpc asymptotic morphism
 $\{\theta_n:C_0(0,1]\otimes D \to \K\}_n$ with $\lim_n\|\theta_n(b)\|=\|b\|$ for all $b$. After passing to a subsequence $(\theta_n)_n$ we may arrange that the maps $\eta_n$
 obtained as the compositions
 \[
 \xymatrix{
	A\ar[r]^-{\gamma_n} &  C_0(0,1]\otimes D \ar[r]^-{j\otimes \id{D}} & C_0(0,1]\otimes C_0(0,1]\otimes D\ar[r]^-{\id{}\otimes \theta_n} &C_0(0,1]\otimes \K
 }
 \]
 define an asymptotic morphism $\{\eta_n:A \to C_0(0,1]\otimes \K\}_n$ such that \linebreak $\lim_n \|\eta_n(a)\|=\|a\|$ for all $a\in A$.

 We regard $\eta_n : A \to C\K$ as a continuous family of maps $\{\eta^{(t)}_n : A \to \K\}_{t\in [0,1]}$ with $\eta^{(0)}_n=0$.  Next we want to arrange that  $\lim_n\|\eta^{(1)}_n(a)\|=\|a\|$ for all $a\in A$.
 Since  $\lim_n\|\eta_n(a)\|=\|a\|$ for all $a\in A$, after passing to a subsequence of $(\eta_n)_n$, we may arrange for each $n\geq 1$,
 $\|\eta_n(a_i)\|>\|a_i\|-1/i$ for all $1\leq i \leq n$. Therefore for every $n$ and $i$ with $1\leq i \leq n$, there is $t_{n,i}\in [0,1]$ such that $\|\eta^{(t_{n,i})}_n(a_i)\|>\|a_i\|-1/i$.
 Define $E_n:A \to C\K$ by $E_n^{(t)}=\bigoplus_{i=1}^n \eta^{(tt_{n,i})}_n$ for $t\in [0,1]$.
 It follows immediately that $\{E_n:A \to C\K\}_n$ is a cpc asymptotic morphism such that $\lim_n\|E^{(1)}_n(a)\|=\|a\|$ for all $a\in A$.
\end{proof}

\begin{definition}\label{def:strong_AHZ}
(i) A separable $C^*$-algebra $A$ has \emph{property \textrm{(QH)}} if it satisfies one of the equivalent conditions from Proposition~\ref{prop:QH-equivalence}.

(ii) For a countable discrete group $G$, denote the character induced by the trivial
representation by $\iota \colon C^*(G) \to \C$ and set $I(G) = \ker(\iota)$. We
say that $G$ has \emph{property $\textrm{(QH)}$} if $I(G)$ has property \textrm{(QH)}.
One can reformulate this condition as follows.
There is a discrete injective ucp asymptotic morphism $\{\pi_n:C^*(G)\to M_{m(n)}(\C)\}_n$
which is homotopic to $\{\iota_n:C^*(G)\to M_{m(n)}(\C)\}_n$, where each $\iota_n$ is the multiple
$m(n) \cdot\iota$ of the trivial representation $\iota$.
\end{definition}

\begin{example} If $X$ is a compact connected metrizable space and $x_0\in X$, then $C_0(X\setminus x_0)$
has property \textrm{(QH)}. In particular, if $G$ is a torsion free countable discrete abelian group, then $I(G)\cong C_0(\widehat{G}\setminus \iota)$ has property \textrm{(QH)}.
\end{example}
\begin{remark} \label{rem:qd}
Let us note that if $A$ has  property \textrm{(QH)}, then $A$ must be quasidiagonal by a result of Voiculescu \cite{Voi:qd} and $A$ cannot have any nonzero projections. In particular, if a discrete group $G$ has property \textrm{(QH)},
then $G$ must be torsion free, see Remark~\ref{remark:final}.
It is easily verified that
\begin{itemize}
\item[(i)] Property \textrm{(QH)} passes to $C^*$-subalgebras.
\item[(ii)] If a separable $C^*$-algebra $A$ has property \textrm{(QH)} then so does $A\otimes_{min} B$
for any separable $C^*$-algebra $B$.
\end{itemize}
Other permanence properties are proved in Theorem~\ref{thm:permanence}.

\end{remark}

The following Lemma is crucial for the proof that property (QH) implies the existence
of additive inverses  of homotopy classes of discrete asymptotic morphisms. It is a slight generalization
of \cite[Lem.\ 5.3]{Dad-IJM}.
The proof is almost identical, except that one has to replace Voiculescu's theorem
and Stinespring's theorem by Kasparov's generalization of these
\cite{Kas:cp}.
We will use the notation from \cite[Sec.\ 5]{Dad-IJM}: If $A$ and $B$ are $C^*$-algebras,
$E$, $F$ are right Hilbert $B$-modules, $\fset \subset A$ is a finite set, $\epsilon > 0$ and
$\varphi \colon A \to \bdd{B}{E}$ and $\psi \colon A \to \bdd{B}{F}$ are two maps,
we write $\varphi \fleq{\fset}{\epsilon} \psi$ if there is an isometry $v \in \bdd{B}{E,F}$
such that $\lVert \varphi(a) - v^*\psi(a)v \rVert < \epsilon$ for all $a \in \fset$. If $v$
can be chosen to be a unitary, we write $\varphi \fsim{\fset}{\epsilon} \psi$. Moreover, we
write $\varphi \prec \psi$ if $\varphi \fleq{\fset}{\epsilon} \psi$ for all finite sets
$\fset$ and for all $\epsilon > 0$.

\begin{lemma} \label{lem:discrete_inverse}
Let $A$ and $B$ be separable unital $C^*$-algebras such that $A$ or $B$ is nuclear.
Let  $\{\varphi_n \colon A \to M_{k(n)}(B)\}_n$
and $\{\gamma_n \colon A \to M_{r(n)}(\C)\}_n$ be ucp discrete asymptotic morphisms. Suppose that
$(\gamma_n)$ is injective. Then there exist a sequence $(\omega(n))$ of disjoint
finite subsets of $\N$ with $\max \omega(n-1) < \min \omega(n)$ and a ucp discrete asymptotic
morphism  $\{\varphi^{\,\prime}_n \colon A \to M_{s(n)}(B)\}_n$ such that
$\left(\varphi_n \oplus \varphi^{\,\prime}_n\right) \sim (\gamma_{\omega(n)} \otimes 1_B)$, where
$\gamma_{\omega(n)} = \oplus_{i \in \omega(n)} \gamma_i$.
\end{lemma}

\begin{proof} \label{pf:discrete_inverse}
There exists a sequence $\fset_1 \subseteq \fset_2 \subseteq \dots$ of finite selfadjoint
subsets of $A$ consisting of unitaries such that their union is dense in $U(A)$ and a
sequence $\epsilon_1 \geq \epsilon_2 \geq \dots$ convergent to zero such that $\varphi_n$
is $(\fset_n, \epsilon_n)$-multiplicative. By \cite[Lem.\ 5.1]{Dad-IJM} it suffices to
construct a sequence $(\omega(n))$ such that $\varphi_n \fleq{\fset_n}{\epsilon_n}
\gamma_{\omega(n)}$.

We will construct $(\omega(n))$ inductively. Suppose that we already have $\omega(1), \dots,
\omega(n-1)$ and choose $m > \max \omega(n-1)$ such that
\[
	\Gamma = \oplus_{i \geq m} \gamma_i \colon A \to \prod_{i \geq m} M_{r(i)}(\C) \subset \mathcal{L}(H)
\]
is $(\fset_n, \tfrac{\epsilon_n^2}{4})$-multiplicative. Observe that $\Gamma \colon A \to \mathcal{L}(H)$ is a unital $*$-mo\-no\-mor\-phism modulo the compact operators.
Let $\pi \colon A \to \mathcal{L}(H)$ be a faithful unital representation with $\pi(A) \cap \mathcal{K}(H) = \{0\}$, then we have $\Gamma \fsim{\fset_n}{\epsilon_n} \pi$ by \cite[Lem.\ 2.1]{Brown-Herrero} (or \cite[Lem.\ 5.2]{Dad-IJM}).
Note that there is a natural embedding $\mathcal{L}(H) \to \bdd{B}{\ltwo{B}}$, which sends $T$ to $T \otimes 1_B$.
Moreover, $\Gamma \otimes 1_B \fsim{\fset_n}{\epsilon_n} \pi \otimes 1_B$. Let $\pi_n$ be the Stinespring dilation of $\varphi_n$ obtained via \cite[Thm.\ 3]{Kas:cp}.
We have that $\pi_n \oplus (\pi \otimes 1_B) \sim (\pi \otimes 1_B)$ by \cite[Thm.\ 6]{Kas:cp}. Thus,
\[
	\varphi_n \prec \pi_n \prec \pi_n \oplus (\pi \otimes 1_B) \sim (\pi \otimes 1_B) \fsim{\fset_n}{\epsilon_n} \Gamma \otimes 1_B\ ,
\]
hence $\varphi_n \fleq{\fset_n}{\epsilon_n} \Gamma \otimes 1_B$. Let $v \in \bdd{B}{B^{k(n)}, \ltwo{B}}$ be a partial isometry such that $\lVert \varphi_n(a) - v^*(\Gamma \otimes 1_B)(a)v \rVert < \epsilon_n$ for all $a \in \fset_n$. The projections of $M_{\infty}(B)$ are dense in those of $B \otimes \mathcal{K}(H)$. Thus, we can find an isometry $w \in \bdd{B}{B^{k(n)},B^N}$ which approximates $v$ in norm sufficiently well so that
$\lVert \varphi_n(a) - w^*(\Gamma \otimes 1_B)(a)w \rVert < \epsilon_n$ for all $a \in \fset_n$.
Here we identify $B^N \subset \ltwo{B}$ as the submodule of the first $N$ coordinates. It follows that if we let $\omega(n) = \{m, m+1, \dots, m+N\}$, then $\varphi_n \fleq{\fset_n}{\epsilon_n} \gamma_{\omega(n)}$.
\end{proof}

\begin{lemma}\label{lemma:group structure}
Let $\theta:M\to G$ be a surjective morphism of monoids. If $G$ is a group and  all the elements of $\theta^{-1}(1_G)$ are invertible in $M$, then $M$ is a group.
\end{lemma}
\begin{proof}
  For $x\in M$, choose $x'\in M$ such that $\theta(x')=\theta(x)^{-1}$. Then $\theta(xx')=\theta(x'x)=1_G$. It follows that  both $xx'$ and $x'x$ are invertible and hence there are $y_1,y_2\in M$ such that $x(x'y_1)=1_M=(y_2x')x$.  Thus $x$ is invertible.
\end{proof}

\begin{proposition} \label{prop:dam=group}
Let $A$, $B$ be separable $C^*$-algebras such that either $A$ or $B$ is nuclear. If $A$ has property \textrm{(QH)}, then $[[A, B \otimes \K]]^{\rm cp}_{\N}$ is a group.
\end{proposition}

\begin{proof}
We will first prove the proposition for unital $C^*$-algebras $B$.  Let  $\{\varphi_n \colon A \to M_{r(n)}(B)\}_n$ be a discrete cpc asymptotic morphism. By Remark~\ref{rem:compacts_matrices} it suffices to construct an additive inverse of $[[\varphi_n]]$.
Since $A$ has property \textrm{(QH)}, there exists an injective cpc discrete asymptotic morphism  $\{\eta_n \colon A \to M_{s(n)}(\C)\}_n$ which is null-homotopic.

Let $\widetilde{A}$  denote the unitization of $A$  and set $R(n) = r(n)+1$.
 Let $\widetilde{\varphi}_n \colon \widetilde{A} \to M_{R(n)}(B)$ be the unital extension of $\varphi_n$, so that
  $\widetilde{\varphi}_n(1) = 1_{M_{R(n)}} \otimes 1_B$. This is a ucp asymptotic morphism. Likewise, let $\widetilde{\eta}_n \colon \widetilde{A} \to M_{S(n)}(\C)$ for $S(n) = s(n)+1$ be the unitization of $\eta_n$. Note that $(\widetilde{\eta}_n)$ is still injective. From Lemma~\ref{lem:discrete_inverse} we obtain sequences $(\omega(n))$ and $\{\widetilde{\varphi}^{\,\prime}_n \colon A \to M_{T(n)}(B)\}_n$ such that $(\widetilde{\varphi}_n \oplus \widetilde{\varphi}^{\,\prime}_n) \sim (\widetilde{\eta}_{\omega(n)} \otimes 1_B)$. Let $j \colon A \to \widetilde{A}$ be the inclusion map and set $\varphi^{\,\prime}_n = \widetilde{\varphi}^{\,\prime}_n \circ j$.
   Then $(\varphi_n \oplus \varphi^{\,\prime}_n) \sim (\eta_{\omega(n)} \otimes 1_B)$ which is null-homotopic.

Now consider the case when $B$ is nonunital. Observe that for any short exact sequence of separable $C^*$-algebras
\begin{equation}\label{short:exact}0 \to I \to B \stackrel{p}\to D \to 0\end{equation}
and an arbitrary separable $C^*$-algebra $A$ there is a corresponding long exact Puppe sequence of pointed sets:
\begin{equation} \label{eqn:Puppe}
\xymatrix@C=0.5cm{
	[[A, SB]]_{\N}^{\rm cp} \ar[r] & [[A, SD]]_{\N}^{\rm cp} \ar[r] & [[A, C_p]]_{\N}^{\rm cp} \ar[r] & [[A, B]]_{\N}^{\rm cp} \ar[r] & [[A, D]]_{\N}^{\rm cp}
}
\end{equation}
where $C_p = \left\{ (b,f) \in B \oplus C_0([0,1),D)\ |\ f(0) = p(b) \right\}$ is the mapping cone of the $*$-homomorphism $p$.
The proof of exactness is entirely similar to the proof  for the Puppe sequence in $E$-theory, see \cite[Prop. 6]{Dadarlat:anote}.
 Indeed as shown in the proof of \cite[Thm.~3.8]{Rosenberg:puppe}, the mapping cone of the map $C_p\to B$, $(b,f)\mapsto b$ is homotopic as a $C^*$-algebra to $SD$.

If the short exact sequence \eqref{short:exact} splits, the first and the last map in \eqref{eqn:Puppe} are surjective and we obtain the short exact sequence
\[
\xymatrix@C=0.8cm{
	0 \ar[r] & [[A, C_p]]_{\N}^{\rm cp} \ar[r] & [[A, B]]_{\N}^{\rm cp} \ar[r] & [[A, D]]_{\N}^{\rm cp} \ar[r] & 0.
} 	
\]
Moreover, it was proven in \cite[Prop.~3.2]{DadLor:unsusp} that if the short exact sequence \eqref{short:exact} splits, the canonical homomorphism $I \to C_p$ has an inverse in the category $\mathit{Asym}$ of separable $C^*$-algebras and homotopy classes of (non-discrete) asymptotic morphisms.
 On the other hand   \cite[Thm.~7.2]{Thomsen:discrete} shows that an isomorphism  $B_1\cong B_2$ in the category $\mathit{Asym}$ induces an isomorphism $[[A,B_1]]_{\N}\cong [[A,B_2]]_{\N}$ and hence $[[A,B_1]]_{\N}^{\rm cp}\cong [[A,B_2]]_{\N}^{\rm cp}$ if $A$ is nuclear.

 If $B$ is nuclear, then so are $D$, $I$ and $C_p$ and hence
$I$ is isomorphic to $C_p$ in the category $\mathit{Asym}^{\rm cp}.$  By using the version
of \cite[Thm.~7.2]{Thomsen:discrete} for cpc asymptotic morphisms that was mentioned in Remark~\ref{rem:category},
we see that $[[A,C_p]]_{\N}^{\rm cp}\cong [[A,I]]_{\N}^{\rm cp}$ if $B$ is nuclear.
Therefore if either $A$ or $B$ is nuclear,  the following is a short sequence of pointed sets:
\begin{equation}\label{eq:splitt}
\xymatrix@C=0.5cm{
0\ar[r] & [[A, I]]_{\N}^{\rm cp} \ar[r] & [[A, B]]_{\N}^{\rm cp} \ar[r] & [[A, D]]_{\N}^{\rm cp}\ar[r] &0.
}
\end{equation}

In the case of the split extension $0 \to B \otimes \K \to \widetilde{B} \otimes \K \to \K \to 0,$ we obtain from  \eqref{eq:splitt} the following short exact sequence of pointed sets
\[
\xymatrix@C=0.8cm{
	0 \ar[r] & [[A, B \otimes \K]]_{\N}^{\rm cp} \ar[r]^-{\alpha} & [[A, \widetilde{B} \otimes \K]]_{\N}^{\rm cp} \ar[r]^-{\beta} & [[A, \K]]_{\N}^{\rm cp} \ar[r] & 0.
} 	
\]
Note that $[[A, B \otimes \K]]_{\N}^{\rm cp}$ is an abelian monoid, and by the first part of the proof, $[[A, \widetilde{B} \otimes \K]]_{\N}^{\rm cp}$ and $[[A, \K]]_{\N}^{\rm cp}$ are abelian groups. All monoids are pointed by their respective neutral elements, the map $\alpha$ is a monoid homomorphism and $\beta$ is a group homomorphism. Exactness implies
that the sequence $0 \to [[A, B \otimes \K]]_{\N}^{\rm cp} \to \ker(\beta)\to 0$ is also exact.
By Lemma~\ref{lemma:group structure} we conclude that $[[A, B \otimes \K]]_{\N}^{\rm cp}$ is an abelian group.
\end{proof}

\begin{proposition} \label{prop:QH_vs_group}
Let $A$ be a separable $C^*$-algebra.

\noindent (i) $A$ has property \textrm{(QH)} if and only if $A$ is quasidiagonal and $[[A, \K]]_{\N}^{\rm cp}$ is a group.

\noindent (ii) If $A$ is nuclear, then $A$ has property \textrm{(QH)} if and only if $[[A, A \otimes \K]]_{\N}$ is a group.
\end{proposition}

\begin{proof}
(i)  If $A$ has property \textrm{(QH)}, then $A$ is quasidiagonal by \cite{Voi:qd}. Proposition~\ref{prop:dam=group} implies that $[[A, \K]]^{\rm cp}_{\N}$ is a group.  For the other direction, since $A$ is quasidiagonal, there is an injective cpc discrete asymptotic morphism
$\{\varphi_n:A \to \K\}_n$. Let $(\varphi^{\,\prime}_n)$ be such that $[[\varphi^{\,\prime}_n]]=-[[\varphi_n]]$ in  $[[A,\K]]_{\N}^{\rm cp}$.
Then $(\eta_n)=(\varphi_n\oplus \varphi^{\,\prime}_n)$ is injective and null-homotopic.

(ii) If $A$ is nuclear, then  $[[A,B]]^{\rm cp}\cong [[A,B]]$ for any separable $C^*$-algebra $B$.  Suppose that $A$ has property \textrm{(QH)}.
Proposition~\ref{prop:dam=group} implies that $[[A,A \otimes \K]]_{\N}$ is a group.

 For the other direction note that since $[[A, B \otimes \K]]_{\N} \cong [[A \otimes \K, B \otimes \K]]_{\N}$ we may assume that $A\cong A \otimes \K$.
Since $[[\id{A}]]$ has an additive inverse in $[[A, A]]_{\N}$ it follows
that there is an injective cpc discrete asymptotic morphism  $\{\varphi_n:A \to A\}_n$ which is null-homotopic. By composing $\varphi_n$ with a representation of $A$ we find  an injective cpc discrete asymptotic morphism  $\{\theta_n:A \to L(H)\}_n$ which is null-homotopic. We conclude by applying
Proposition~\ref{prop:QH-equivalence}.
\end{proof}

\section{Nuclear homotopy symmetric $C^*$-algebras}
\begin{theorem} \label{thm:QH_implies homotopy symmetric}
Let $A$ be a separable, nuclear $C^*$-algebra. Then $A$ has property \textrm{(QH)} if and only if $A$ is homotopy symmetric.  In either case,  $[[A, B \otimes \K]]\cong E(A,B)\cong KK(A,B)$ for any separable $C^*$-algebra $B$.
\end{theorem}
\begin{proof} Suppose first that $A$ has property \textrm{(QH)}.
  By Lemma 5.6 of \cite{Thomsen:discrete} for any separable $C^*$-algebra $B$ there is an exact sequence of pointed sets
  \[
  \xymatrix@C=0.8cm{
	 [[A, SB \otimes \K]]_{\N} \ar[r]^-{\alpha} & [[A, B \otimes \K]] \ar[r]^-{\beta} & [[A,  B \otimes\K]]_{\N} \ar[r]^{1-\sigma} &  [[A,  B \otimes\K]]_{\N}.
} 	
  \]
  Here $\sigma$ is the shift map $\sigma[[\psi_n]]=[[\psi_{n+1}]]$, $\beta$ is the natural restriction map and $\alpha$ is defined by stringing together the components of a discrete asymptotic morphism
  $\{\varphi_n:A \to C_0(0,1)\otimes B \otimes \K\}_n$ to form a continuous asymptotic morphism
  $\{\Phi_t :A \to B \otimes \K\}_{t\in [0,\infty)}$, where  $\Phi_t(a)=\varphi_n(a)(t-n)$ for $t\in [n,n+1]$.
Recall that if the addition operation is defined via direct sums, then  $[[A, B \otimes \K]]$ is an abelian monoid and all the other entries are abelian groups by Proposition~\ref{prop:dam=group}.
  It follows that $(1-\sigma)$ is a morphism of groups and both $\alpha$ and $\beta$ are monoid homomorphisms. By Lemma~\ref{lemma:group structure}, the exact sequence $[[A, SB \otimes \K]]_{\N} \to [[A, B \otimes \K]] \to \ker(1-\sigma) \to 0$ implies
 that $[[A, B \otimes \K]]$ is a group. In particular, taking $B=A$ we see that $A$ is homotopy symmetric.

 Conversely, suppose that  $A$ is homotopy symmetric. Then $[[A\otimes \K,A\otimes \K]]$ is a group.
 The product $[[A,A\otimes \K]]_{\N}\times [[A\otimes \K,A\otimes \K]]\to [[A,A\otimes \K]]_{\N}$, $(x,y)\mapsto y\circ x$ has the property that $(y_1+y_2)\circ x=y_1\circ x+y_2\circ x$. By applying this property
 with $y_1=[[\id{A\otimes \K}]]$ and $y_2=-y_1$ we obtain that  $[[A,A\otimes \K]]_{\N}$ is a group.
 It follows that  $A$ has property \textrm{(QH)} by Proposition~\ref{prop:QH_vs_group}(ii).

  For the last part of the statement we apply the main result of \cite{DadLor:unsusp}.
\end{proof}
\begin{remark}
Let $M$ be a homotopy associative and homotopy commutative $H$-space
and let $X$ be a topological space with the homotopy type of a CW-complex.
By a slight extension of \cite[Thm.\ X.2.4]{Whitehead:top-book} the homotopy classes of
continuous maps $[X,M]$ form a group if and only if $\pi_0(M) = [pt, M]$ is a group.
The combination of Proposition \ref{prop:QH_vs_group} and Theorem~\ref{thm:QH_implies homotopy symmetric}
 provides a counterpart for discrete asymptotic morphisms of this statement: Let $A$ be a nuclear  quasidiagonal $C^*$-algebra.
  The monoid $[[A,B \otimes \K]]$ is a group for any separable $C^*$-algebra $B$ if and only if $[[A, \K]]= [[A, C(pt) \otimes \K]]$
   is a group, if and only if $[[A, \K]]_\N$ is a group.
\end{remark}

The importance of Theorem ~\ref{thm:QH_implies homotopy symmetric}  comes from the fact that
 property \textrm{(QH)}  is much easier to verify than the property of being homotopy symmetric.
 It allows us to vastly extend the class of known homotopy symmetric C*-algebras.

\begin{theorem} \label{thm:permanence}
The class of homotopy symmetric $C^*$-algebras has the following permanence properties:
\begin{enumerate}[{\rm (a)}]
	\item \label{it:subalgebras} A nuclear $C^*$-subalgebra of a separable $C^*$-algebra with property \textrm{(QH)} is homotopy symmetric.
	\item \label{it:gensubalgebras} Let $(A_n)_n$ be a sequence of separable $C^*$-algebras with property \textrm{(QH)}. Any separable nuclear $C^*$-subalgebra  of $\prod_n A_n/\bigoplus_n A_n$
 is homotopy symmetric.
 	\item \label{it:ind-limit} The class of separable nuclear homotopy symmetric $C^*$-algebras is closed under inductive limits.
	\item \label{it:2of3} If $\ 0 \to I \to A \to B \to 0$ is a split short exact sequence of separable nuclear $C^*$-algebras and two of the entries are homotopy symmetric, then so is the third.
	\item\label{it: tensor-homotopy} The class of homotopy symmetric $C^*$-algebras is closed  under tensor products by separable $C^*$-algebras  and under (asymptotic) homotopy equivalence.
	\item \label{it:crossed-prod} The class of separable nuclear homotopy symmetric $C^*$-algebras  is closed under crossed products by second countable compact groups.
\end{enumerate}
\end{theorem}

\begin{proof} \label{pf:permanence}
Since property \textrm{(QH)} passes obviously to $C^*$-subalgebras, statement~(\ref{it:subalgebras}) follows from Theorem ~\ref{thm:QH_implies homotopy symmetric}.

Let $A$ be a $C^*$-subalgebra of $\prod_n A_n/\bigoplus_n A_n$ as in statement~(\ref{it:gensubalgebras}).
Using the Choi-Effros lifting theorem, we find a cpc discrete asymptotic morphism $\{\theta_n:A\to A_n\}_n$
such that $\limsup_n \|\theta_n(a)\|=\|a\|$ for all $a\in A$. Since $A$ is separable, by replacing $\theta_n$ by finite direct sums of the form $\theta_n\oplus \theta_{n+1}\oplus\cdots \oplus \theta_{N}$ and $A_n$ by
$A_n\oplus A_{n+1}\oplus\cdots \oplus A_{N}$ we may assume that
$\lim_n \|\theta_n(a)\|=\|a\|$ for all $a\in A.$ Here we use the observation that the class of $C^*$-algebras with property \textrm{(QH)} is closed under finite direct sums.
Let $(\Phi_n^{A_i})_n$ be the homotopy of discrete asymptotic morphisms given by Proposition~\ref{prop:QH-equivalence}(i) for $A_i$.
Since $A$ is separable, one can find an increasing sequence $m(n)$ of natural numbers such that $\Phi_n:=\Phi_{m(n)}^{A_n}\circ \theta_n$ satisfies the condition (i) of Proposition~\ref{prop:QH-equivalence} for $A$
in the sense that $(\Phi_n)$ is a homotopy between an injective cpc discrete asymptotic morphism and
the null map. We conclude the proof of (\ref{it:gensubalgebras}) by applying Theorem ~\ref{thm:QH_implies homotopy symmetric}.

Statement (\ref{it:ind-limit}) follows from (\ref{it:gensubalgebras}) since any inductive limit $\varinjlim A_n$ embeds as a $C^*$-subalgebra of $\prod_n A_n/\bigoplus_n A_n$.

For the proof of (\ref{it:2of3}) first note that the cases of nuclear subalgebras and quotients follow from (\ref{it:subalgebras}), since we assumed that the sequence splits. It remains to be proven that $A$ is homotopy symmetric if $I$ and $B$ are. If $B$ is homotopy symmetric, then $[[B \otimes \K, B \otimes \K]]$ is a group. There is an element $y \in [[B \otimes \K, B \otimes \K]]$ such that $[[\id{B\otimes \K}]] + y = 0$. Therefore $[[A, B \otimes \K]]$ is a group as well, since $y \circ x$ is an additive inverse of $x \in [[A, B \otimes \K]]$.
 Likewise $[[A, I \otimes \K]]$ is a group if $I$ is homotopy symmetric. By \cite[Prop.3.2]{DadLor:unsusp}  we have a short exact sequence of monoids
\[
	0 \to [[A, I \otimes \K]] \to [[A, A \otimes \K]] \to [[A, B \otimes \K]] \to 0
\]
and Lemma~\ref{lemma:group structure} implies that $[[A, A \otimes \K]]$ is a group. Thus, $A$ is homotopy symmetric.

The statement  (\ref{it: tensor-homotopy}) is an immediate consequence of the definition as noted in \cite{DadLor:unsusp}.

For the proof of (\ref{it:crossed-prod}) let $A$ be a separable $C^*$-algebra and let $G$ be a second countable compact group that acts on $A$ by automorphisms.
Then $$A \rtimes G  \subset \left(A\otimes C(G)\right) \rtimes G \cong A \otimes \K(L^2(G))$$ by \cite[Cor.~2.9]{Green}. We conclude the proof by applying~(\ref{it:subalgebras}).
\end{proof}

The following corollary exhibits a new large class of homotopy symmetric $C^*$-algebras.
\begin{corollary}\label{cor:fields}
 Let $A$ be a separable continuous field of nuclear $C^*$-algebras over a compact connected metrizable space $X$.
If one of the fibers of $A$ is homotopy symmetric, then $A$ is homotopy symmetric.
\end{corollary}
\begin{proof}
 $A$ has nuclear fibers and hence it is a nuclear $C^*$-algebra.
$A$ embeds in $C(X)\otimes \mathcal{O}_2$ by \cite{Blanchard:subtriviality}.
Fix $x_0\in X$ such that $A(x_0)$ is homotopy symmetric. Furthermore, since the embedding is $C(X)$-linear, it follows that $A$ embeds in $E=\{f\in C(X)\otimes \mathcal{O}_2: f(x_0)\in D\}$, where $D \subset \Cuntz{2}$ is a $C^*$-subalgebra isomorphic to $A(x_0)$. Thus, $E$ fits into a short exact sequence
\[
\xymatrix{
	0 \ar[r] & C_0(X \setminus \{x_0\}) \otimes \Cuntz{2} \ar[r] & E \ar[r]^-{ev_{x_0}} & D \ar[r] & 0
}
\]
This sequence splits via the $*$-homomorphism $D \to E$ that maps $d$ to the constant function $f(x)=d$ on $X$. Since both $C_0(X \setminus\{x_0\}) \otimes \Cuntz{2}$ and $D$ are homotopy symmetric, the statement  now follows from Theorem~\ref{thm:permanence} (\ref{it:2of3}).
\end{proof}

\begin{remark} \label{rem:unsusp-annalen}
Statement (\ref{it:ind-limit}) in Theorem~\ref{thm:permanence} strengthens the main result of \cite{Dad:unsusp-annalen} in the case of nuclear $C^*$-algebras.
\end{remark}

\section{Group $C^*$-algebras}
In this section we prove that any countable torsion free nilpotent group has property \textrm{(QH)}\!.\begin{theorem} \label{central-extensions}
Let $1\to N\to G \to H \to 1$ be a central extension of discrete countable amenable groups where $N$ is torsion free. If $H$ has
property \textrm{(QH)} then so does $G$.
\end{theorem}
\begin{proof} By
 \cite[Thm. 1.2]{Packer-Raeburn:twisted_groups}, (see also \cite[Lemma 6.3]{Echterhoff_Williams:Crossed_products}),
  $C^*(G)$ is a nuclear continuous field
of $C^*$-algebras over the spectrum $\widehat{N}$ of $C^*(N)$. Moreover,
the fiber over the trivial character $\iota$ of $N$ is isomorphic to $C^*(H)$. It follows that $I(G)$ is a nuclear continuous field
of $C^*$-algebras over the spectrum $\widehat{N}$ whose fiber at $\iota$ is isomorphic to $I(H)$. Since $N$ is torsion free, its Pontriagin dual
is connected. We conclude the proof by applying Cor.~\ref{cor:fields}.
\end{proof}
\begin{lemma}\label{lemma:union-AH}
Suppose that a countable  discrete amenable group $G$ is the union of an increasing sequence of subgroups $(G_i)_i$ each of which
has property \textrm{(QH)}\!. Then $G$ has property \textrm{(QH)}\!.
\end{lemma}
\begin{proof} Since $G$ is amenable, so is each $G_i$ and the associated group $C^*$-algebras  are nuclear.
We may regard  $C^*(G_i)$ as a $C^*$-subalgebra of $C^*(G)$.
Then the union of $C^*(G_i)$ is dense in $C^*(G)$ and hence the union of $I(G_i)$ is dense in $I(G)$.
The conclusion follows from  Thm.~\ref{thm:permanence}~(\ref{it:ind-limit}).
\end{proof}
\begin{theorem}\label{thm:AH2}
If $G$ is a countable torsion free nilpotent group, then $I(G)$ is a homotopy symmetric $C^*$-algebra.
\end{theorem}
\begin{proof}
Subgroups of nilpotent groups are nilpotent.
Consequently, we may assume by Lemma~\ref{lemma:union-AH} that $G$ is finitely generated.
Since $G$ is nilpotent it has a finite upper central series $(Z_i)_{i=0}^n$ consisting of subgroups
\begin{equation}\label{eqn:ups}\{1\}=Z_0\subset Z_1 \subset \cdots \subset Z_{n-1}\subset Z_n=G,\end{equation}
where $Z_1$ is the center of $G$ and for $i\geq 1$, $Z_{i+1}$ is the unique subgroup of $G$ such that $Z_{i+1}/Z_i$ is the center of $G/Z_i$.
We argue by induction on the length $n$ of the central series of $G$.
Suppose that  property \textrm{(QH)} holds for all finitely generated nilpotent groups with upper central series of length $n-1$. If $G$ is a finitely generated, torsion free and satisfies ~\eqref{eqn:ups}, then $G/Z_1$ is finitely generated, torsion free and nilpotent  by \cite[Cor. 1.3]{Jennings} and $(Z_{i+1}/Z_1)_{i=0}^{n-1}$ is a central series of length $n-1$ for $G/Z_1$.
Since the upper central series is the shortest central series \cite[5.1.9]{Robinson},
we conclude the proof by applying Theorem~\ref{central-extensions} to the central extension
$ 1\to Z_1\to G\to G/Z_1\to 1.$
  \qedhere\end{proof}

\begin{remark}\label{remark:final} (i) The assumption that $G$ is torsion free is essential. Indeed if $s\in G$ is an element of order $n>1$, then
$1-\frac{1}{n}(1+s+\cdots +s^{n-1})$ is a nonzero projection contained in $I(G)$.

(ii) If $\pi: \mathbb{H}_3 \to U(n)$ is a representation of the Heisenberg group whose restriction to the center
is non-trivial, then there is no continuous path of representations connecting $\pi$  to a multiple of the trivial representation, \cite{Adem2015}.

(iii) The K-homology of nilpotent groups such as $\mathbb{H}_{2n+3}$, $n\geq 1$, has nontrivial torsion, \cite{Lee-Packer}.
In view of Theorem~\ref{thm:AH2}, there are matricial deformations of $C^*(\mathbb{H}_{2n+3})$ which detect this torsion.
\end{remark}


\end{document}